\documentclass[11pt]{amsart}
\usepackage{pxfonts}
\usepackage{amsmath}
\usepackage{amssymb}
\usepackage{amsthm}
\usepackage{amscd, amsfonts}
\usepackage[dvips]{epsfig}
\usepackage{verbatim}

\usepackage{epsf}
\usepackage[usenames]{color}
\usepackage[bookmarksnumbered,pdfpagelabels=true,plainpages=false,colorlinks=true,
            linkcolor=black,citecolor=black,urlcolor=blue]{hyperref}

\newcommand{\al}{\alpha}
\newcommand{\be}{\beta}
\newcommand{\ga}{\gamma}

\newcommand{\ep}{\epsilon}

\newcommand{\tr}{\operatorname{tr}}

\theoremstyle{plain}
\newtheorem{theorem}{Theorem}[section]
\newtheorem{lemma}[theorem]{Lemma}
\newtheorem{prop}[theorem]{Proposition}

\theoremstyle{definition}
\newtheorem{rema}[theorem]{Remark}
\newtheorem{defi}[theorem]{Definition}

\numberwithin{equation}{section}

\title[A priori estimates for nonconcave PDE]{$C^{2,\al}$ estimates and existence results for certain nonconcave PDE}

\author[Pingali]{Vamsi P. Pingali}
\address{Department of Mathematics\\
412 Krieger Hall, Johns Hopkins University, \\ Baltimore, MD 21218, USA}
\email{vpingali@math.jhu.edu}

\begin{document}
\maketitle
\begin{abstract}
We establish $C^{2,\al}$ estimates for PDE of the form convex $+$ a sum of weakly concave functions of the Hessian, thus generalising a recent result of Collins which is in turn inspired by a theorem of Caffarelli and Yuan. Independently, we also prove an existence result for a certain generalised Monge-Amp\`ere PDE.
\end{abstract}

\section{Introduction}
\indent In the classic paper \cite{Kryl} Krylov studied the following PDE on a convex domain.
\begin{gather}
S_m (D^2 u) = \displaystyle \sum _{k=0} ^{m-1} (l^{+}_k)^{m-k+1}(x) S_k (D^2 u) 
\label{Kryeq}
\end{gather}
where $S_m (A)$ is the $m$th elementary symmetric polynomial of the symmetric matrix $A$. He proved that the corresponding Dirichlet problem has a smooth solution in the ellipticity cone of the equation.
This was accomplished by reducing the equation to a Bellman equation and then using the standard theory of Bellman equations. Motivated by complex-geometric considerations (Chern-Weil theory) a very special case of equation \ref{Kryeq} was studied in \cite{GenMA} and an existence result was proven using the method of continuity. To this end, \emph{a priori} estimates on the solution were necessary. The $C^{2,\al}$ estimate for such nonlinear PDE is usually given by the Evans-Krylov-Safonov theorem which applies to PDE of the form $F(D^2 u) =0$ where $F$ is a concave function of symmetric matrices. However, it is not immediately obvious that equation \ref{Kryeq} is concave. Yet, upon dividing by $\det(D^2 u)$ and rearranging the equation one can see that it is actually concave and thus amenable to Evans-Krylov theory. \\
\indent Unfortunately, not all PDE can be rewritten to be concave functions of the Hessian. Indeed, not all level sets have a positive second fundamental form. To remedy this partially, Caffarelli and Yuan \cite{Caff} proved a result that roughly speaking, allows one of the eigenvalues of the second fundamental form of the level set of $F(D^2 u)$ to be negative. Using similar ideas, Cabre and Caffarelli  \cite{Caff2} proved $C^{2,\al}$ estimates for functions that are the minimum of convex and concave functions. Even these theorems cannot handle the following PDE that arises in the study of the J-flow on toric manifolds \cite{Col1} \footnote{Actually, the Legendre transform of the solution occurs in the $J$-flow.}.
\begin{gather}
\det(D^2 u) + \Delta u = 1.
\label{natural}
\end{gather}
Moreover, equation \ref{natural} is also a real example of a ``generalised Monge-Amp\`ere" PDE introduced in \cite{GenMA}. \\
\indent In \cite{Col1} Collins and Sz\'ekelyhidi proved interior $C^{2,\al}$ estimates for equation \ref{natural} using ideas from \cite{Caff}. In \cite{Col2} Collins generalised that result to obtain the following theorem. (The precise definition of ``twisted" type equations is recalled in section \ref{prel}.).
\begin{theorem}{(Collins)}
Consider the equation $F(D^2 u, x) = F_{\cup} (D^2 u, x) + F_{\cap} (D^2 u, x)=0$ on the unit ball $B_1$ in $\mathbb{R}^n$. For each $x$, assume that $F$ is of the twisted type. Let $0<\lambda < \Lambda <\infty$ be ellipticity constants for both $F, F_{\cup}$. For every $0<\al<1$ we have the estimate
\begin{gather}
\Vert D^2u \Vert _{C^{\al}(B_{1/2})}\leq C(n,\lambda, \Lambda, \al, \ga, \Gamma, \Vert F_{\cup}\Vert _{L^{\infty}(D^2 u (\bar{B_1}))}, \Vert F_{\cap} \Vert _{L^{\infty}(D^2 u (\bar{B_1}))},\Vert D^2 u \Vert _{L^{\infty} (B_1)}),
\end{gather}
where $0<\gamma = \inf_{x \in F_{\cup}(D^2 u)(B_1)} G^{'}(-x)$ and $\Gamma = osc_{B_1} G(-F_{\cup}(D^2 u))$. ($G$ is defined in section \ref{prel}.)
\label{Colthm}
\end{theorem}
Motivated by these developments, in this paper we prove the following improvement of Collins' result.
\begin{theorem}
Consider the equation $F(D^2 u, x) = F_{\cup} (D^2 u, x) + \displaystyle \sum _{\al=1}^m F_{\cap, \al} (D^2 u, x) = 0$ on the unit ball $B_1$ in $\mathbb{R}^n$. For each $x$, assume that $F$ is of the ``generalised" twisted type. Let $0<\lambda < \Lambda <\infty$ be ellipticity constants for both $F, F_{\cup}$. For every $0<\al<1$ we have the estimate
\begin{gather}
\Vert D^2u \Vert _{C^{\al}(B_{1/2})}\leq C(n,\lambda, \Lambda, \al, \ga, \Vert F_{\cup}\Vert _{L^{\infty}(D^2 u (\bar{B_1}))}, \Vert F_{\cap} \Vert _{L^{\infty}(D^2 u (\bar{B_1}))},\Vert D^2 u \Vert _{L^{\infty} (B_1)}, \Vert G \Vert _{L^{\infty}(W)}),
\end{gather}
where $0<\gamma = \inf_{\{x \in W \} } G^{'}(x)$ and $W=\displaystyle \bigcup _{\al =1} ^{m}F_{\cap, \al} (D^2 u (\bar{B}_1)) \bigcup _{1\leq j\leq m} \bigcup _{\{ x \ \in \ \bar{B}(1) \}}  \sum_{\al =1 }^{j}  F_{\cap, \al}(D^2 u (x))$.
\label{estimates}
\end{theorem}
The proof of theorem \ref{estimates} follows the arguments (with some modifications) in \cite{Col2, Caff}. Independently, we also prove the following existence result.
\begin{prop}
Consider the following PDE,
\begin{gather}
\det(D^2 u) + \displaystyle \sum _{k=2} ^n S _k (D^2 u) = f \ on \ D \nonumber \\
u \vert_{\partial D} = \phi,
\label{locequa}
\end{gather}
where $S _k$ is the $k$th symmetric polynomial (for instance $\sigma _n$ is the determinant), $f: \bar{D} \rightarrow (n-1, \infty)$ and $\phi$ are smooth functions (with $\phi$ being the restriction to $\partial D$ of a smooth function on $\bar{D}$), and $D$ is a strictly convex domain with a proper smooth defining function $\rho$, i.e., $\rho ^{-1} (0) = \partial D$, $\rho ^{-1}(-\infty,0) = D$, $\nabla \rho \neq 0 \ on \ \partial D$, and $D^2 \rho \geq C I$ ($C>0$ is a constant). It has a unique smooth solution $u$ such that $D^2 u > -I$ and $\frac{\partial}{\partial \lambda _i} (\lambda _1\lambda _2 \ldots \lambda _n + \displaystyle \sum _{k=2} ^n \sigma _k (\vec{\lambda})) > 0 \ \forall \ i$ where $\lambda _i$ are the eigenvalues of $D^2 u$.
\label{locthm}
\end{prop}
The requirement $f>n-1$ is not optimal. But we give a counterexample for finding solutions in the ellipticity cone in the case $f<0$. Notice that this seemingly harder equation has an existence result but it is still not clear whether equation \ref{natural} does. \\
\indent The layout of the paper is as follows. In section \ref{prel} we give the definitions of twisted type equations and give an example of its applicability. In section \ref{easy} we prove proposition \ref{locthm} and discuss its hypotheses.\\
\emph{Acknowledgements} : The author thanks Professor Joel Spruck for his suggestions and Tristan Collins for answering queries about his paper.
\section{Preliminaries}\label{prel}
In this section we present the definitions and prove some basic results. \\
Firstly, we define what it means for a PDE to be of the generalised twisted type. The following definition generalises Collins' \cite{Col2}.
\begin{defi}
Let $F(D^2u)=0$ be a uniformly elliptic equation on the unit ball $B_1$. It is said to be of the generalised twisted type if $F = F_{\cup} + \displaystyle \sum _{\al =1}^m F_{\cap, \al}$ where
\begin{enumerate}
\item $F_{\cup}$ and $\forall \ 1\leq\al\leq m \ F_{\cap,\al}$ are (possibly degnerate) elliptic $C^2$ functions on an open set containing $D^2 u (\bar{B_1})$.
\item $F_{\cup}$ is convex and uniformly elliptic, and $\displaystyle \sum _{\al =1}^m F_{\cap ,\al}$ is weakly concave in the sense of definition \ref{weco}.
\end{enumerate}
\label{gentwist}
\end{defi}
The definition of weak concavity in our case is as follows.
\begin{defi}
We say that  $\displaystyle \sum _{\al =1}^m F_{\cap ,\al}$ is weakly concave if there exists a function $G : U\rightarrow \mathbb{R}$ such that
\begin{enumerate}
\item The domain $U$ contains a connected open set $V$ with compact closure containing $W=\displaystyle \bigcup _{\al =1} ^{m}F_{\cap, \al} (D^2 u (\bar{B}_1)) \bigcup _{1\leq j\leq m} \bigcup _{\{ x \ \in \ \bar{B}(1) \}}  \sum_{\al =1 }^{j}  F_{\cap, \al}(D^2 u (x))$.
\item $G^{'} >0$, $G^{''}\leq 0$, and $G(F_{\cap,\al}(.))$ is concave for all $1\leq \al \leq m$.
\item For all $x \ \in \ \bar{B}(1)$ and $1\leq \al \leq m$ consider $y_{\al}(x)=F_{\cap, \al} (D^2 u (x))$. There exists a constant $1\geq c>0$ independent of $x$ such that $\displaystyle  \sum _{i=1} ^m G(y_i(x))\geq G\left(\sum _{i=1}^{m} y_i(x)\right) \geq c\sum _{i=1}^{m}  G( y_i(x))$.
\end{enumerate}
\label{weco}
\end{defi}
Definition \ref{weco} might seem somewhat convoluted and unnatural compared to the analogous one in \cite{Col2}. Firstly, we remark that condition $(3)$ is actually redundant in many cases of interest (but we choose to impose it since it appears naturally in our proofs). Indeed,
\begin{prop}
Given a function $\tilde{G}$ that satisfies requirements $(1),(2)$ of definition \ref{weco}  such that $W \subseteq \mathbb{R}_{\geq 0}$, automatically satisfies requirement $(3)$, i.e.,
\begin{gather}
\displaystyle  \sum _{\al=1} ^m \tilde{G}(y_{\al}(x)) \geq  \tilde{G}\left(\sum_{al=1} ^m y_{\al} (x) \right) \geq \frac{1}{2^m} \sum _{\al=1} ^m \tilde{G}(y_{\al}(x)).
\end{gather}
\label{redu}
\end{prop}
\emph{Proof}   Consider the function $T(y) = \tilde{G}(y+z)-\tilde{G}(y)-\tilde{G}(z)$ for a fixed $z \geq 0$. By the concavity of $G$ we see that $T^{'}(y) \leq 0$. Hence $\tilde{G}(y+z)-\tilde{G}(y)-\tilde{G}(z) \leq -\tilde{G}(0)=0$. Using induction we see that $\displaystyle\sum _{\al=1} ^m \tilde{G}(y_{\al}(x)) \geq  \tilde{G}\left(\sum_{\al=1} ^m y_{\al}(x) \right)$. The concavity of $G$ implies that $\tilde{G}\left(\frac{y+z}{2}\right) \geq \frac{\tilde{G}(y) + \tilde{G}(z)}{2}$. Since $\tilde{G}$ is increasing this implies that $\tilde{G}(y+z)\geq \frac{\tilde{G}(y) + \tilde{G}(z)}{2}$. Induction gives the desired result. \qed \\
\begin{rema}\label{moregen}
 Furthermore, it is more natural to have a different $G_{\al}$ that works for $F_{\cap, \al}$. However, under mild conditions on such $G_{\al}$ one may produce a $G$ that works for all $1\leq \al \leq m$. Indeed, assume that $\bar{V} \subset \mathbb{R} _{\geq 0}$, and $G_{\al}$ are such that on the appropriate compact sets $G_{\al} \geq 0$, $G_{\al} ^{'}  \geq 1$ and $G_1 (\bar{V}) \subseteq \mathrm{dom}(G_2), \ G_2(G_1(\bar{V})) \subseteq \mathrm{dom}(G_3) \ldots$. \\
\indent Consider the function $H_{k} = G_{k} \circ G_{k-1} \ldots \circ G_1$. Notice that
\begin{gather}
D^2 H_k(F_{\cap, k}) = H_k ^{''} DF_{\cap, k} DF_{\cap, k} + H_k^{'}D^2 F_{\cap, k} \nonumber \\
= (G_k ^{''} (H_{k-1} ^{'})^2 + G_k ^{'} H_{k-1} ^{''}) DF_{\cap, k} DF_{\cap, k} + G_{k} ^{'} H_{k-1} ^{'} D^2 F_{\cap, k} \nonumber
\end{gather}
Inductively we may assume that $H_{k-1} ^{'} \geq 1$. Thus we get
\begin{gather}
D^2 H_k(F_{\cap, k}) \leq H_{k-1} ^{'} (G_k ^{''}DF_{\cap,k}DF_{\cap, k} + G_k ^{'}D^2 F_{\cap, k}) + G_k ^{'} H_{k-1} ^{''}DF_{\cap,k}DF_{\cap, k} \leq 0\nonumber
\end{gather}
where we used the facts that $G_{k} \circ F_{\cap, k}$ is concave, $H_{k-1}^{'} > 0$, $G_k ^{'} > 0$, and $H_{k-1}$ is concave. Now notice that if $H$ is any concave increasing function and $Y(A)$ is any concave function of symmetric matrices, then $D^2(H\circ Y) = H^{''} DY DY + H^{'}D^2 Y \leq 0$. This means that $H_{m} \circ F_{\cap ,\al}$ is concave for all $1\leq \al \leq m$. Using proposition \ref{redu} we are done.
\end{rema}
\indent Now we give an example of an equation that satisfies the conditions imposed by theorem \ref{estimates}.
\begin{prop}\label{sat}
Consider the following equation on a domain $\Omega$.
\begin{gather}
H(D^2 u,x) = \tr (AD^2u) + \sum_{k=2}^{n} f_k \sigma_{k,B_k} (D^2 u) = g 
\label{main}
\end{gather}
where $g:\bar{\Omega} \rightarrow \mathbb{R}_{> 0}, f_k : \bar{\Omega} \rightarrow \mathbb{R}_{\geq 0}$ are smooth functions. Also assume that $A, B_k$ are smooth, positive-definite $n\times n$ real matrix-valued functions on $\bar{\Omega}$.  $\sigma_{k,B}(A)$ be the coefficient of $t^k$ in $\det(B+tA)$. Equation \ref{main} is of the generalised twisted type on every ball $B_r(x_0) \subseteq \Omega$ if $D^2 u > 0$ on $\bar{\Omega}$.
\end{prop}
\begin{proof}
Fix an $x$. In equation \ref{main} $F_{\cup} (D^2 u) = \tr(AD^2u)$ which is obviously smooth and uniformly elliptic. As for $F_{\cap, \al}(D^2 u) = \sigma _{\al, B_{\al}} (D^2 u)$, firstly by means of diagonalising the quadratic form $B_{\al}$ we may assume that it is the identity matrix. Thus, at the point $x$ we see that $F_{\cap, \al}(D^2 u)$ is a positive multiple of the $\al$th symmetric polynomial. Hence it is elliptic if $C I  > D^2 u > 0$ \footnote{It may not be uniformly elliptic because we don't have a given lower bound on $D^2 u$, but that is not a requirement anyway.}. Therefore $F(D^2 u)$ is uniformly elliptic. Moreover, the function $G(x) = x^{1/n}$ defined on $\mathbb{R}_{>0}$ satisfies the conditions required by definition \ref{weco}. Indeed, since $(\sigma _{k, B_k}) ^{1/k}$ is concave it is clear that $(\sigma _{k, B_k}) ^{1/n}$ is too.
\end{proof}
\section{Proof of theorem \ref{estimates}}\label{higher}
\indent As mentioned in the introduction we prove a stronger version of theorem \ref{Colthm}, i.e. instead of $F_{\cup} + F_{\cap} = 0$ we have $F_{\cup} +\displaystyle \sum _{\al = 1} ^m F_{\cap,\alpha} = 0$ where there exists a $G$ so that $G(F_{\cap, \alpha})$ is concave for every $\alpha$.
\indent The strategy to prove theorem \ref{estimates} is exactly the one used in \cite{Caff, Col1, Col2}. Here is a high-level overview:
\begin{enumerate}
\item One first reduces the content of theorem \ref{estimates} to the case where $F(D^2 u,x)$ does not depend on $x$. Indeed, one can use a blowup argument \`a la \cite{Col2} to conclude this. This reduction step requires $F$ to be uniformly elliptic which it is by assumption.
\item In the case of $F(D^2 u)=0$, one proves that the level set of $u$ is very ``close" to a quadratic polynomial satisfying $F(D^2 P)=0$ (after ``zooming" in so to say). This is done by proving that $F_{\cup} (D^2 u)$ concentrates in measure near its level set, and using the Alexandrov-Bakelmann-Pucci estimate in conjunction with the usual Evans-Krylov theory to conclude the existence of a polynomial close to $u$. Then one perturbs the polynomial to make it satisfy $F(D^2 P)=0$.
\item Then it may be proven that one can find a family of such quadratic polynomials with the ``closeness" improving in a quantitative way on the size (the smaller the better) of the neighbourhood of the point in consideration.
\item This can be used to prove that the second derivative does not change too much, i.e., the desired estimate on $\Vert D^2 u \Vert _{C^{\al}(B_{1/2})}$.
\end{enumerate}
Out of these, only step $2$ needs modification in our case.  To this end, we need the following lemma.
\begin{lemma}
Let $L$ be the linearisation of $\displaystyle F = F_{\cup} + \sum _{\al} F_{\cap , \al}$, i.e. $\displaystyle L^{ab} = F_{\cup} ^{ab} + \sum _{\al} F_{\cap, \al} ^{ab}$. Then
$$\displaystyle L\left(\sum _\al G(F_{\cap, \al} (D^2 u))\right) \leq 0.$$
\label{subso}
\end{lemma}
\begin{proof}
We may compute
\begin{gather}
\partial _a G (F_{\cap, \al} (D^2 u)) = G^{'} F_{\cap, \al} ^{ij} u_{x_ax_ix_j} \nonumber \\
\partial _{ab} G (F_{\cap, \al} (D^2 u)) = G^{''} F_{\cap, \al} ^{ij} u_{x_ax_ix_j}F_{\cap, \al} ^{rs} u_{x_bx_rx_s} + G^{'} F_{\cap, \al} ^{ijrs} u_{x_ax_ix_j}u_{x_bx_rx_s} + G^{'} F_{\cap , \al} u_{x_ax_bx_ix_j}. \nonumber \\
\end{gather}
Moreover, using the equation itself we obtain,
\begin{gather}
L^{ab} u_{x_ax_bx_i} = (F_{\cup} ^{ab}+ \sum _{\al} F_{\cap, \al} ^{ab}) u_{x_ax_bx_i} = 0 \nonumber \\
L^{ab} u_{x_ax_bx_ix_j}+(F_{\cup} ^{ab rs} + \sum _{\al} F_{\cap, \al} ^{abrs}) u_{x_ax_bx_i}u_{x_rx_sx_j} = 0.
\end{gather}
Then we get
\begin{gather}
\displaystyle L\left(\sum _{\al=1} ^m G(F_{\cap, \al} (D^2 u))\right) = \sum _{\al=1} ^m L^{ab} (G^{''} F_{\cap, \al} ^{ij} u_{x_ax_ix_j}F_{\cap, \al} ^{rs} u_{x_bx_rx_s} + G^{'} F_{\cap, \al} ^{ijrs} u_{x_ax_ix_j}u_{x_bx_rx_s} + G^{'} F_{\cap , \al} ^{ij} u_{x_ax_bx_ix_j}) \nonumber \\
= \sum _{\al=1} ^m L^{ab} (G^{''} F_{\cap, \al} ^{ij}F_{\cap, \al} ^{rs} +G^{'} F_{\cap, \al} ^{ijrs})  u_{x_ax_ix_j}u_{x_bx_rx_s} + G^{'} L^{ab} F_{\cap , \al}^{ij} u_{x_ax_bx_ix_j} \nonumber \\
= \sum _{\al=1} ^m \left ( (F_{\cup} ^{ab}+ \sum _{\be} F_{\cap, \be} ^{ab}) (G^{''} F_{\cap, \al} ^{ij}F_{\cap, \al} ^{rs} +G^{'} F_{\cap, \al} ^{ijrs})  u_{x_ax_ix_j}u_{x_bx_rx_s} - G^{'} F_{\cap , \al}^{ab} (F_{\cup} ^{ij rs} + \sum _{\be} F_{\cap, \be} ^{ijrs}) u_{x_ix_jx_a}u_{x_rx_sx_b}\right ) \label{eq} \\
= \sum _{\al=1} ^m \left (  F_{\cup} ^{ab} (G^{''} F_{\cap, \al} ^{ij}F_{\cap, \al} ^{rs} +G^{'} F_{\cap, \al} ^{ijrs})  u_{x_ax_ix_j}u_{x_bx_rx_s} + \sum _{\be} F_{\cap, \be} ^{ab} G^{''} F_{\cap, \al} ^{ij}F_{\cap, \al} ^{rs} u_{x_ix_jx_a}u_{x_rx_sx_b} - G^{'} F_{\cap , \al}^{ab} F_{\cup} ^{ij rs} u_{x_ix_jx_a}u_{x_rx_sx_b} \right ) \label{eq2}
\end{gather}
At this point we note that since $G\circ F_{\cap ,\al}$ is concave and $F_{\cup}$ is elliptic the first term in \ref{eq2} is negative. Likewise, so is the second term because $G^{''}\leq 0$ and $F_{\cap}$ is also elliptic. Since $F_{\cup}$ is convex, so is the third term. Hence we see that
$$\displaystyle L\left(\sum _\al G(F_{\cap, \al} (D^2 u))\right) \leq 0.$$
Note that in equation \ref{eq} the terms of the form $F^{ab}_{\cap, \al} F^{ijrs}_{\cap, \be}$ cancelled out. This is perhaps the main point of this calculation. If we had different $G_{\al}$ for each $\al$ this would not have happened.
\end{proof}
Secondly, we need the following proposition that actually addresses step $2$ in the strategy described above.
\begin{prop}
Under the assumptions of the main theorem, for any given $\ep >0$ there exists a positive constant $\eta = \eta (c,m, \Vert G \Vert _{L^{\infty}}, \Vert F_{\cap , \al}\Vert _{L^{\infty}},n,\lambda, \Lambda, \epsilon, \gamma, \Gamma, \Vert D^2 u \Vert _{L^{\infty}})$ quadratic polynomial $P$ so that for all $x$ in $B_1$,
\begin{gather}
\vert \frac{1}{\eta^2} u(\eta x) - P(x) \vert \leq \ep \nonumber \\
F(D^2 P) = 0 \nonumber
\end{gather}
\label{close}
\end{prop}
\begin{proof}
We shall determine $k_0, \rho, \xi, \delta$ in the course of the proof. Let $1\leq k \leq k_0$ and $t_k = \max _{\bar{B}(1/2^k)} F_{\cup} (D^2 u)$ and $s_k = \displaystyle \min _{\bar{B}(1/2^k)} \sum _{\al =1}^m G(F_{\cap ,\al} (D^2 u))$. Also define $w_k (x) = 2^{2k} u(\frac{x}{2^k})$. Hence $D^2 w_k (x) = D^2 u (\frac{x}{2^k})$. \\
Note that since $G$ is increasing, $G(-t_k) = G\left(\min _{\bar{B}(1/2^k)} \displaystyle \sum _{\al =1}^m F_{\cap, \al} (D^2 u)\right) = \min _{\bar{B}(1/2^k)} G\left(\displaystyle \sum _{\al =1}^m F_{\cap, \al} (D^2 u)\right) \geq cs_k $. Moreover, $s_k \geq G(-t_k)$.\\
\indent If there exists an $l$ such that $1\leq l \leq k_0$ such that
\begin{gather}
 \vert E_k \vert \leq \delta \vert B_{1/2^l}\vert
\label{condition}
\end{gather}
where $E_k$ is the set of $x$ such that $F_{\cup}$ is ``close" to $t_k$, i.e. $F_{\cup} (D^2u) \leq t_k - \xi$, then we are done by the arguments of \cite{Col2}. If not, we shall arrive at a contradiction by actually proving the existence of such a $\delta$, $k$ and $l$. Indeed, assume the contrary. By lemma \ref{subso} we see that $L\left(\displaystyle \sum _{\al} G(F_{\cap, \al}(D^2 w_k)) - s_k \right) \leq 0$. By applying the weak Harnack inequality we see that for all $x$ in $B_{1/2}$
\begin{gather}
\displaystyle \sum _{\al} G(F_{\cap, \al}(D^2 w_k))(x) - s_k \geq C(n,\lambda) \Vert \sum _{\al} G(F_{\cap, \al}(D^2 w_k))(x) - s_k \Vert _{L^{p_0}(B_1)},
\end{gather}
where $p_0$ depends on $n, \lambda, \Lambda$. On $E_k$ we recall that $\displaystyle \sum _{\al} F_{\cap, \al}(D^2 w_k) \geq -t_k + \xi$, and hence $\displaystyle \sum _{\al} G(F_{\cap, \al}(D^2 w_k)) \geq G\left(\displaystyle \sum _{\al} F_{\cap, \al}(D^2 w_k)\right)  \geq G(-t_k +\xi) \geq G(-t_k) + \gamma \xi \geq c s_k + \gamma \xi$. Choose $\xi$ to be large enough so that $(c-1) s_k + \gamma \xi \geq \theta_0 > 0$ where $\theta_0$ does not depend on $k$. Of course such a $\theta_0$ would depend on $\Vert D^2 u \Vert _{L^{\infty} (B_1)}$, $\Vert F_{\cap, \al} \Vert_{L^{\infty}}$, and $\Vert G\Vert_{L^{\infty}}$. This means that
\begin{gather}
\displaystyle \sum _{\al} G(F_{\cap, \al}(D^2 w_k))(x)  \geq s_k +  C(n,\lambda) \theta _0 \delta ^{1/p_0} = s_k +\theta \nonumber
\end{gather}
In particular this means that $s_{k+1} \leq s_k + \theta$. At this point it follows that after $k_0 = \frac{\mathrm{Osc}_{B_1} (\sum_{\al} F_{\cap ,\al} (D^2 u))}{\theta}$ iterations condition \ref{condition} ought to hold.
\end{proof}
The rest of the proof of theorem \ref{estimates} is exactly the same as in \cite{Caff}. \\
\section{Proof of proposition \ref{locthm}}\label{easy}
We reduce theorem \ref{locthm} to Krylov's equation \ref{Kryeq} and invoke the existence result in \cite{Kryl}. Indeed, define $v = u + \frac{1}{2} \displaystyle \sum _{i=1} ^n x_i ^2$. Then $D^2 v = D^2 u + I$. The eigenvalues of $D^2 v$ are $\mu_i = \lambda _i +1$. Consider the equation
\begin{gather}
\mu_1 \mu _2 \ldots \mu _n - \displaystyle \sum _{i=1} ^n \mu _i = f-n+1 \ on \ D \nonumber \\
v\vert_{\partial D} = \phi + \frac{1}{2} \displaystyle \sum _{i=1} ^n x_i ^2.
\label{neweq}
\end{gather}
Writing equation \ref{neweq} in terms of $\lambda _i$ we see quite easily that equation \ref{locequa} is recovered. Thus, Krylov's theorem \cite{Kryl} states that there is a unique smooth solution to \ref{neweq} in the ellipticity cone as long as the right hand side is positive. This proves proposition \ref{locthm}. \qed \\
\indent As mentioned in the introduction, the restriction $f > n-1$ may not be optimal (as is easily seen by considering a radial solution in the case of the ball with a constant $f$). However, the following counterexample shows that the case $f<0$ does not admit solutions in the ellipticity cone. 
\begin{prop}
There is no smooth solution $u$ of the following equation satisfying $\mu _1 \ldots \mu _{i-1} \mu _{i+1} \ldots \mu _n > 1$ and $\mu_i>0$ where $\mu _i$ are the eigenvalues of $D^2 v$.
\begin{gather}
\det (D^2 v) - \Delta v = -c \ in \ B(1) \nonumber \\
v\vert_{\partial B(1)} = 0 
\label{countereq}
\end{gather} 
where $c>n-1$ is a constant.
\label{counter}
\end{prop}
\begin{proof}
We first show that such a solution has to be radially symmetric. To this end, we use the standard method of moving planes \cite{Evans}. For $0\leq t \leq 1$ consider the plane $P_{t} : x_n = t$. Let the reflection of the point $x$ across the plane $P_{t}$ be $x_{t} = (x_1,\ldots,x_{n-1},2t-x_n)$ and let $E_{t} = \{ x \in B(1) \vert t <x_n \leq 1 \}$. We prove that 
$$u(x) > u(x_{t}) \ \forall \ x \ \in \ E_t \ (property \ (L)).$$
  Near any boundary point the function is strictly increasing as a function of $x_n$ because $\frac{\partial u}{\partial n} \geq 0$ and $D^2 u >0$. Hence (L) holds for $t<1$ sufficiently close to $1$. Let the infimum of all such $t$ be $t_0$. If $t_0 >0$, then consider $w(x)=u(x)-u(x_{t_0})$ where $x \in E_{t_0}$. Upon subtracting the equations for $u(x)$ and $u(x_{t_0}$ we see that
\begin{gather}
\det(D^2 u(x)) - \Delta (u(x)) - (\det(D^2 u (x_{t_0}))-\Delta u(x_{t_0})) = 0 \nonumber \\
\Rightarrow \displaystyle \int _{0} ^1 \frac{d}{ds} (\det(D^2 (s u(x) +(1-s)u(x_{t_0})))-\Delta (s u(x) +(1-s)u(x_{t_0}))) = 0 \nonumber \\
\Rightarrow L^{ij} w_{ij} (x) = 0,
\label{mov}
\end{gather}
where $L^{ij}$ is a positive definite matrix depending on $u$. Note that we have used the assumption that $D^2 u$ is in the ellipticity cone and the fact that the cone is convex for this equation. Since $w\geq 0$ in $E_{t_0}$ and $w=0$ on the plane $P_{t_0}$, by applying the strong minimum principle we see that $w>0$ in $E_{t_0}$. Applying the Hopf lemma to points on the plane $P_{t_0}$  we see that $w_{x_n} > 0$ on $P_{t_0}\cap B(1)$. Since $w_{x_n} = 2 u_{x_n}$ on the plane, we see that for $t$ slightly less than $t_0$ property (L) holds. This is a contradiction. Thus $t_0=0$. Since the problem is rotationally symmetric, $u$ is radial. The unique radial solution to the problem (if it exists) is easily seen to be of the form $\frac{A(r^2-1)}{2}$ for some constant $A>0$. This means that $A^n - nA +c =0$. It is easy to see that this equation admits no positive solutions. 
\end{proof}

\end{document}